\documentclass[letterpaper, 10 pt, conference]{ieeeconf}  

\IEEEoverridecommandlockouts                              

\title{\LARGE \bf
Solvability regions of affinely parameterized quadratic equations
}

\author{Krishnamurthy Dvijotham$^{1}$, Hung Nguyen$^{2}$ and Konstantin Turitsyn$^{2}$
\thanks{*KD was supported by project 1.10 in the Control of Complex Systems Initiative, a Laboratory Directed Research and Development (LDRD) program at the Pacific Northwest National Laboratory.  KT  and  HN  weresupported by the NSF awards 1554171 and 1550015.}
\thanks{$^{1}$Krishnamurthy Dvijotham is with the Department of Mathematics, Washington State University and with the Optimization and Control group, Pacific Northwest National Laboratory, Richland, WA 99354, USA
        {\tt\small dvij@cs.washington.edu}}%
\thanks{$^{2}$Hung Nguyen and Konstantin Turitsyn are with the Department of Mechanical Engineering, Massachusetts Institute of Technology,
        Cambridge, MA 02139, USA
        {\tt\small hungtd@mit.edu,turitsyn@mit.edu}}%
}

\usepackage{latexsym, graphicx, epsfig, amsmath, amsfonts,amssymb,amsthm,algorithm,algorithmicx,booktabs,bm,color,bigints}
\usepackage{tikz}
\usepackage{subcaption}
\usepackage{wrapfig}
\usetikzlibrary{matrix,chains,positioning,decorations.pathreplacing,arrows}

\def\R{{\mathbb R}}

\def\_#1{{\underline{#1}}}

\newtheorem{theorem}{Theorem}
\newtheorem{lemma}{Lemma}
\newtheorem{definition}{Definition}

\newcommand{\br}[1]{\left({#1}\right)}
\newcommand{\norm}[1]{\left\|{#1}\right\|}

\newcommand{\qud}[3]{\left({#1},{#2},{#3}\right)}

\newcommand{\Ucal}{\mathcal{U}}
\newcommand{\inv}[1]{{\br{#1}}^{-1}}
\newcommand{\herm}[1]{\overline{#1}}
\newcommand{\Com}{\mathbb{C}}
\newcommand{\diag}[1]{\left[\!\left[{#1}\right]\!\right]}
\DeclareMathOperator*{\cupstar}{\cup}
\usepackage{verbatim}
\begin{document}
\maketitle
\thispagestyle{empty}
\pagestyle{empty}

\begin{abstract}Quadratic systems of equations appear in several applications. The results in this paper are motivated by quadratic systems of equations that describe equilibrium behavior of  physical infrastructure networks like the power and gas grids. The quadratic systems in infrastructure networks are parameterized - the parameters can represent uncertainty (estimation error in resistance/inductance of a power transmission line, for example) or controllable decision variables (power outputs of generators, for example). It is then of interest to understand conditions on the parameters under which the quadratic system is guaranteed to have a solution within a specified set (for example, bounds on voltages and flows in a power grid). Given nominal values of the parameters at which the quadratic system has a solution and the Jacobian of the quadratic system at the solution is non-singular, we develop a general framework to construct convex regions around the nominal value such that the system is guaranteed to have a solution within a given distance of the nominal solution. We show that several results from recent literature can be recovered as special cases of our framework, and demonstrate our approach on several benchmark power systems.
\end{abstract} 
\section{Introduction}
In this paper, we study systems of affinely parameterized quadratic equations $f(x,u)=0$ where $f$ is vector-valued function that is quadratic in $x$ and affine in $u$. We are interested in the set of $u$ for which the system of equations $f(x,u)=0$ is guaranteed to have a solution for $x$: we denote this set by  $\Ucal$. Due to the nonlinearity of $f$, $\Ucal$ is nonconvex in general and difficult to characterize exactly. In this paper, we aim to find \emph{tractable (convex) sufficient } conditions on $u$ that guarantee existence of solutions.

This question is motivated by systems of equations that describe the equilibrium behavior of infrastructure systems like the power grid or the gas grid. In the power grid, the physical state $x$ corresponds to complex voltage phasors at every node in the network and the parameters $u$ represent controllable decisions (power generation at conventional gas-fired/coal-fired generators) or uncertain parameters (power generation from solar or wind farms, impedances of transmission lines etc.). In the gas grid, $x$ represents pressures at nodes in the network and $u$ represents compressor gains or gas withdrawal/production. Network operators need to ensure that the steady-state equations have a solution for $x$ (voltages or pressures) within acceptable limits (voltage/pressure limits). Thus, it is of interest to find conditions on $u$ that guarantee the existence of a solution to $f(x,u)=0$ for $x$ within given limits.   The conditions can be used for several applications, including robust feasibility assessment (given an uncertainty set for $u$, assess whether the solution $x$ will remain within limits for all values of $u$ in the uncertainty set) and robust optimization (choosing values for controllable parts $u^c$ of $u=u^c+\omega$ so as to ensure robust feasibility with respect to all realizations of uncertainty $\omega$).

Methods from algebra and constraint programming \cite{arnon1984cylindrical}\cite{goldsztejn2006inner}\cite{franek2015robust}\cite{franek2016robustness} study this problem, but typically rely on a discretization of the state space (simplicial or box decomposition) that grows exponentially with the problem dimension. Algorithms developed in the numerical analysis community \cite{frommer2007proving}\cite{beelitz2009framework} verify the existence of a solution for a given value of $u$ but do not easily extend to computing inner approximations of $\mathcal{U}$. There is work in the field of polynomial optimization \cite{lasserre2009moments} that studies existence of solutions to polynomial equations. The most relevant to our setting is the work in  \cite{lasserre2015tractable}, where a general framework is proposed to study feasible sets of polynomial systems with quantifiers. However, the approach can only compute \emph{outer} approximations of $\mathcal{U}$. Further, the framework requires solution of a large semidefinite programs and is not suitable for fast computations required in applications like the power grid. In previous work \cite{dvijotham2015construction}, we developed a framework based on polynomial optimization to check whether a given set $S$ is contained in $\Ucal$ - however, this framework also requires solving a large SDP to perform the check and further requires $S$ to be pre-specified (in other words it cannot automatically construct $S$). 

Several papers have studied the specific quadratic equations arising in power systems (the AC power flow equations). In \cite{louca2015acyclic}\cite{louca2016stochastic}, the authors propose approaches to develop inner approximations of quadratically constrained quadratic programs arising in power systems. However, their approach cannot handle nonlinear equality constraints (only inequalities). \cite{bolognani2016existence}\cite{wang2016explicit}\cite{yu2015simple} construct explicit convex inner approximations of $\Ucal$ under some conditions. However, the results obtained in these papers can be conservative and cannot handle all the parameters in the power flow equations (they typically only account for power injections in $u$, all other parameters like line impedances etc. are treated as fixed quantities). Further, they are restricted to power distribution systems and do not extend to power transmission systems.

In this paper, we establish a novel framework capable of:
\begin{itemize}
\item[(1)] Producing \emph{explicit convex inner approximations}  of $\Ucal$ (stated as explicit convex conditions and do not require solution of an optimization problem).
\item[(2)] Handling arbitrary quadratic equations with affine parameters appearing in the equations in the quadratic, linear and constant terms.
\item[(3)] Providing tightness guarantees on our inner approximations (i.e, a measure of how far the inner approximations are from the boundary of $\mathcal{U}$).
\end{itemize}

The rest of this paper is organized as follows: In section \ref{sec:MathFormulation}, we formulate the problem mathematically and introduce the relevant background. In section \ref{sec:MainResults}, we present our main results on constructing tractable inner approximations of $\mathcal{U}$ and theoretical results on tightness of the inner approximations.  In section \ref{sec:NumRes}, we present numerical evaluations on a number of IEEE power system benchmarks and in section \ref{sec:Conc}, we wrap up with conclusions and directions for future work.

\section{Background and Mathematical formulation}\label{sec:MathFormulation}
We begin by introducing notation that will be used in this paper - $\R$ denotes the set of real numbers and $\Com$ the set of complex numbers. $\herm{x}$ denotes the conjugate of the complex number $x$. $\mathrm{Int}\br{S}$ denotes the interior of $S$ for any set $S \subseteq \R^n$ and is defined as
\[\left\{x \in S: \exists \epsilon>0 \text{ such that } \{y:\norm{y-x}_2\leq \epsilon\}\subseteq S\right\}\]
We use $\norm{\cdot}$ to denote an arbitrary norm on $\R^n$ (or $\Com^n$). Given a matrix $M\in\R^{n \times n}$ (or $\Com^{n \times n}$). $\norm{M}=\max_{\norm{x}\leq 1}\norm{Mx}$ denotes the induced matrix norm. Specific examples in the paper use the $\infty$ norm:$\norm{x}_\infty=\max_i |x_i|$ for $x \in \R^n$ (or $\Com^n$) and $\norm{M}_\infty=\max_i \br{\sum_j |M_{ij}|}$ for $M\in \R^{n \times n}$ (or $\Com^{n \times n}$).

We study systems of quadratic equations that are parameterized affinely in a parameter $u$:
\begin{align}
f\br{x,u}=\underbrace{Q\qud{x}{x}{u}}_{\text{Quadratic terms}}+\underbrace{L\br{x,u}}_{\text{Linear terms}}+\underbrace{K\br{u}}_{\text{Constant terms}}=0 \label{eq:Quad}
\end{align}
where $Q: \R^n \times \R^n \times \R^k \mapsto \R^n$ is a tri-linear function (linear in each of its inputs for fixed values of the other inputs) with $Q\qud{0}{0}{u}=0 \quad \forall u \in \R^k$ and represents the quadratic terms in $x$, $L:\R^n \times \R^k \mapsto \R^n$ is a bi-linear function with $L(0,u)=0 \quad \forall u\in\R^k$ that  represents the linear terms in $x$, and $K: \R^n\mapsto \R^n$ is a linear function represents the constant terms in $x$. We view \eqref{eq:Quad} as a system of quadratic equations in $x$ parameterized by $u$. Given a \emph{nominal solution} $\br{x^\star,u^\star}$ satisfying \eqref{eq:Quad}, we define conditions on $u$ that guarantee the existence of a solution for $x$ within a certain distance of $x^\star$. The conditions we construct involve the Jacobian of the system $f$ with respect to $x$. In our setup, it is convenient to define the Jacobian via its action on a vector $y\in \R^n$:
\begin{subequations}
\begin{align}
J\br{x,u;y} & =\left.\frac{\partial f}{\partial x}\right|_{x,u}y = 2Q\qud{x}{y}{u}+L\br{y,u}\\
J_\star\br{u;y} & =\left.\frac{\partial f}{\partial x}\right|_{x=x^\star,u}y,J_\star\br{u^\star;y}=J_{\star\star}y
\end{align}
\end{subequations}
where $\br{x^\star,u^\star}$ denote a \emph{nominal solution} of \eqref{eq:Quad}. This captures how the Jacobian changes with $x$ and $u$ and with respect to $u$ for a fixed $x^\star$.

We now consider the specific cases of power flow and gas flow equations and show how they are special cases of the abstract model \eqref{eq:Quad}:
\paragraph{Power flow} The AC power flow equations characterize the steady state of the power grid and can be written as
\begin{align}
\sum_{k=1}^n {\herm{Y_{ik}}}V_i\herm{V_k}+Y_{i0}V_0=s_i,i=1,\ldots,n \label{eq:PF}
\end{align}
where $V_i$ denotes complex voltage phasor at node $i$, $Y$ denotes the admittance matrix and $s$ denotes complex power injections in the grid. The node $0$ is called the \emph{slack bus} and has a fixed reference voltage $V_0$. We can rewrite \eqref{eq:PF} in the form \eqref{eq:Quad} with $x$ denoting the real and imaginary parts of $V$ (since the equation is quadratic in $V$) and $u$ denoting real and imaginary parts of $s$ and $Y$ (since $s$ and $Y$ appear linearly). We seek conditions on $s$ and $Y$ that guarantee existence of a solution $V$ to to \eqref{eq:PF} - this is referred to as \emph{voltage stability} in the power systems literature. Non-existence of a power flow solution leads to \emph{voltage collapse} where the voltages in the grid drop rapidly, causing devices to trip out and leading to a major partial or complete blackout. Thus, it is of interest to ensure that $u$ is such that the system of equations has a solution for $x$.
\paragraph{Gas flow} The steady-state gas-flow equations can be written as follows \cite{DjGas}:
\begin{subequations}\label{eq:GF}\label{eq:GF}
\begin{align}
&q_i=\sum_{k} \phi_{ik} \\
&\alpha_{ik}\br{\pi_i-r_{ij}\lambda_{ij}\psi_{ij}}=\pi_j+(1-r_{ij})\lambda_{ij}\phi_{ij}\psi_{ij} \\
&\psi_{ij}^2=\phi_{ij}^2 \\
&\psi\geq 0,\pi\geq 0
\end{align}
\end{subequations}
where $q$ denotes production or withdrawal of gas, $\pi_i$ the pressure at node $i$, $\phi_{ij}$ denotes flow of gas on pipeline $\br{i,j}$, $\psi=|\phi|$ denotes the absolute value of flow, $r_{ij},\lambda_{ij}$ are parameters of the pipeline and $\alpha$ denotes the compressor gain of a compressor on the pipeline. This can also be put into the form \eqref{eq:Quad} with $x=\br{\pi,\phi,\psi}$ and $u=\br{q,\alpha}$ - again it is of interest to understand under what conditions on $u$ does \eqref{eq:GF} have a solution within limits on $\pi,\psi$ (pressure/flow limits) - contracts between gas operators and local utilities typically require that the pressure is above a minimum threshold and safety considerations require that the pressure is below some maximum limit (beyond which pipelines may leak or even burst). 

The main problem we solve is defined below:
\begin{definition}[Problem Statement]\label{def:ProbStatement}
Given $\br{x^\star,u^\star}$ satisfying \eqref{eq:Quad}, a norm $\norm{\cdot}$ and a positive number $r>0$, construct a region $\Ucal^i_r\subseteq \R^k$ such that
\begin{align*}
& u^\star \in \mathrm{Int}\br{\Ucal^i_r}\\
& \Ucal^i_r \subseteq \Ucal_r=\left\{u: \exists x:\norm{x-x^\star}\leq r \text{ such that } f(x,u)=0\right\}
\end{align*}
\end{definition}
Note that we assume that the constraints on $x$ are always given by a norm ball around the nominal solution. While this may not be true in general, for typical constraints in practical applications (for example, linear constraints $Ax\leq b$), given any $x^\star$ that is \emph{strictly feasible} ($Ax^\star\leq b$), it is easy construct an $r$ such that a norm-ball of radius $r$ around $x^\star$ is contained in $\{x:Ax\leq b\}$.

Our main tool for computing inner approximations of $\Ucal$ is Brouwer's fixed-point theorem, which provides a sufficient condition for the existence of solutions to nonlinear systems within convex sets.
\begin{theorem}[Brouwer's fixed-point theorem]\label{thm:Brouwer}
Let $F: \R^n \mapsto \R^n$ (or $\Com^n \mapsto \Com^n$) be a continuous mapping and $S \subset \R^n$ (or $\Com^n$) be a compact convex set. Then, if
$F\br{y} \in S \quad \forall y \in S $, $F\br{y}=y$ has a solution in $S$.
\end{theorem}

\section{Inner approximations of the solvability region}\label{sec:MainResults}
In this section, we state our main theoretical results characterizing inner approximations of $\Ucal$. We then discuss tightness of these inner approximations (how far they are from the boundary of $\Ucal$) and how they relate to previous results in the literature. We begin by stating our main result: Given any nominal solution to the quadratic system with non-singular Jacobian, we can construct a family of convex regions satisfying the requirements of definitions of \ref{def:ProbStatement}. We then discuss tightness guarantees for our certificate (showing that under certain assumptions, that our sufficient conditions are ``close'' to necessary) and compare our results to previous work in the context of AC power flow and finally discuss computational issues.

Our main result shows that given \emph{any} solution of $\br{x^\star,u^\star}$ of $f\br{x,u}=0$ with non-singular Jacobian $J_{\star\star}$, we can construct a family of convex regions containing $u^\star$ for which existence of a solution within distance $r$ from $x^\star$ is guaranteed. We start by providing a brief outline of the proof for a simple case to illustrate the arguments behind the theorem. Consider a quadratic system of the form $Q(x,x)+x+u=0$ with $x^\star=0,u^\star=0$ as the nominal solution. Our strategy is to apply Brouwer's fixed point theorem (theorem \ref{thm:Brouwer}) to the fixed point system $x=T(x)=-(Q(x,x)+u)$. We try to prove $T$ maps a convex set of the form $\{x:\norm{x}\leq r\}$ onto itself. In order to do this, we need to prove that
$\norm{Q(x,x)+u}\leq r \quad \forall x:\norm{x}\leq r$, or \\
$\displaystyle\max_{\norm{x} \leq r}\norm{Q(x,x)+u}\leq r^2 \br{\displaystyle\max_{\norm{x}\leq 1} \norm{Q(x,x)}}+\norm{u} \leq r$ \\
where we used the triangle inequality and the fact that $Q\br{rx,rx}=r^2Q\br{x,x}$. 
Dividing by $r$ throughout, we get \\
$\displaystyle\max_{\norm{x} \leq r}\norm{Q(x,x)+u}\leq r \br{\displaystyle\max_{\norm{x}\leq 1} \norm{Q(x,x)}}+\frac{\norm{u}}{r} \leq 1$.\\
This establishes a convex condition on $u$ that guarantees existence of a solution in the set $\{x: \norm{x}\leq r\}$. The following lemma that generalizes the above argument:
\begin{lemma}\label{lem:MainCert}
Let $\br{x^\star,u^\star}$ satisfy \eqref{eq:Quad} and suppose that $J_{\star\star}$ is invertible. Define the following quantities:
\begin{align*}
e\br{u-u^\star} & =\norm{\inv{J_{\star\star}}\br{f\br{x^\star,u}-f\br{x^\star,u^\star}}}\\
h\br{u} & =\max_{\norm{y}\leq 1} \norm{ \inv{J_{\star\star}}Q\qud{y}{y}{u}}, \\
 g\br{u-u^\star} & =\max_{\norm{y}\leq 1} \norm{\inv{J_{\star\star}}J_\star\br{u;y}-y}.
\end{align*}
Then for any $r>0$, \eqref{eq:Quad} has a solution in the set $\{x:\norm{x-x^\star}\leq r\}$ if
\begin{align}
r h\br{u}+ g\br{u-u^\star}+\frac{1}{r}e\br{u-u^\star}	\leq 1 \label{eq:rd}
\end{align}
\end{lemma}
\begin{proof}
See appendix section \ref{sec:ProofMainCert}.
\end{proof}

The above results produce a convex condition on $u$ for each value of $r>0$. However, it is not true that the condition being satisfied for $r^\prime>0$ is implied by the condition for $r>r^\prime$. Thus, if one is interested in existence of a solution in the set $\{x:\norm{x-x^\star}\leq r\}$, it makes sense to take the union of the regions defined by plugging in $r\delta$ into \eqref{eq:rd} for each $\delta\in (0,1]$. Further, if one is simply interested in existence of a solution (rather than existence within a particular set), its possible to optimize the value of $r$ so as to minimize the LHS of \eqref{eq:rd} for any given value of $u$. By doing this, we obtain the maximal set of $u$ for which existence of a solution is guaranteed. This leads to the following theorem, which is our main result on existence of solutions to \eqref{eq:Quad}:
\begin{theorem}\label{thm:MainCert}
Define the sets
\begin{subequations}\label{eq:Cert}
\begin{align}
&\Ucal^{i}_r(\delta)=\left\{u: r\delta h\br{u}+ g\br{u-u^\star}+\frac{1}{r\delta}e\br{u-u^\star}	\leq 1 \right\} \label{eq:CertConvex}\\
&\Ucal^i_r = \displaystyle\cupstar_{0< \delta \leq 1} \Ucal^i_r(\delta)
\end{align}
\end{subequations}
Then for each $u \in \Ucal^i_r$, there is a solution to \eqref{eq:Quad} with $\norm{x-x^\star}\leq r$. Further, $\Ucal^i_r(\delta)$ is a convex set and hence $\Ucal^i$ is a union of convex sets and $\Ucal^i$ has non-empty interior with $u^\star \in \mathrm{Int}\br{\Ucal^i_r}$. Further, the union $\cupstar_{0\leq r<\infty} \Ucal^i_r$ is given by:
\begin{align}
&\Ucal^{i}_\infty=\left\{u: 2\sqrt{h\br{u}e\br{u-u^\star}}+ g\br{u-u^\star}\leq 1 \right\} \label{eq:CertFinal}
\end{align}
and for each $u \in \Ucal^i_\infty$ there exists a solution to \eqref{eq:Quad}.
\end{theorem}
\begin{proof}
See appendix section \ref{sec:ProofMainCert}.
\end{proof}

We now interpret the terms that appear in condition \eqref{eq:CertConvex}: \\
(1) $h(u)$ is a measure of how ``nonlinear'' the system is, since it bounds the \emph{gain of the quadratic term}, or the maximum norm of $Q\br{y,y,u}$ given inputs $y$ within the unit ball $\norm{y}\leq 1$. Thus, the more nonlinear the system is, the more difficult it is to satisfy \eqref{eq:CertConvex}. \\
(2) $J_\star\br{u;y}$ is equal to $\frac{\partial f}{\partial x} y$ evaluated at $\br{x^\star,u}$ and we have $f\br{x^\star+y,u}\approx f\br{x^\star,u}+J_\star\br{u;y}$.
       $J_{\star\star}$ represents the Jacobian of the system evaluated at the nominal solution $\br{x^\star,u^\star}$. Rescaling by the inverse of $J_{\star\star}$ ensures that our conditions are affine-invariant (invariant to an affine scaling of the equations). Thus, $\inv{J_{\star\star}}\br{f\br{x^\star,u}+J_\star\br{u;y}}\approx \inv{J_{\star\star}}f\br{x^\star+y,u}$.
       The terms $g\br{u-u^\star},e\br{u-u^\star}$ bound the deviation of the LHS of the above expression from the identity map $y \mapsto y$: $\inv{J_{\star\star}}\br{f\br{x^\star,u}+J_\star\br{u;y}}-y$.
       If $u=u^\star$, this deviation is $0$ and increases as $u$ moves away from $u^\star$ - thus, the quantities $e,g$ bound how much the first-order linear approximation of the nonlinear equations changes as a function of $u$.

\subsection{Tightness of the certificate} \label{sec:tightness}
In this section, we analyze the tightness of theorem \ref{thm:MainCert} for the special case when the quadratic and linear terms do not depend on $u$ Thus \eqref{eq:Quad} simplifies to:
\begin{align}
Q\br{x,x}+Lx+K\br{u}=0\label{eq:QuadSimple}
\end{align}
\begin{theorem}\label{thm:Converse}
Let $\br{x^\star,u^\star}$ satisfy \eqref{eq:QuadSimple}. Define $e,g,h$ as in theorem \ref{thm:MainCert} and let $h^\star=h\br{u}$ (since $Q$ is independent of $u$, so is $h$).
For any $\kappa \in (0,1)$, the system of equations \eqref{eq:QuadSimple} has a solution in the set $\mathcal{B}_\kappa=\left\{x:\norm{x-x^\star} \leq \frac{\kappa}{2h^*}\right\}$ if $u$ belongs to the set
\begin{align}
\mathcal{U}^i_\kappa=\left\{u:e\br{u-u^\star} \leq  \frac{2\kappa-\kappa^2}{4h^\star}\right\} \label{eq:CertSimple}
\end{align}
Conversely, if $u$ lies outside the set
\begin{align}
\Ucal^o_\kappa=\left\{u:e\br{u-u^\star} \leq \frac{2\kappa+\kappa^2}{4h^\star}\right\}\label{eq:CertSimpleNecc}
\end{align}
then \eqref{eq:QuadSimple} does not have a solution in $\mathcal{B}_\kappa$. Thus,
$\Ucal^i_\kappa \subseteq \Ucal_\kappa \subseteq \Ucal^o_\kappa$.
\end{theorem}
\begin{proof}
See appendix section \ref{App:PfConverse}
\end{proof}
This result shows that the inner approximation $\mathcal{U}^i_\kappa$ is tight upto a factor,
in the sense that for any $u$ outside $\Ucal^o_\kappa$ (which is simply $\Ucal^i_\kappa$ scaled by a factor $\frac{2+\kappa}{2-\kappa}\leq 3$) cannot have a solution in $\mathcal{B}_\kappa$ (see figure \ref{fig:InnerOuter}). \begin{figure}[htb]
\centering \includegraphics[width=.4\columnwidth]{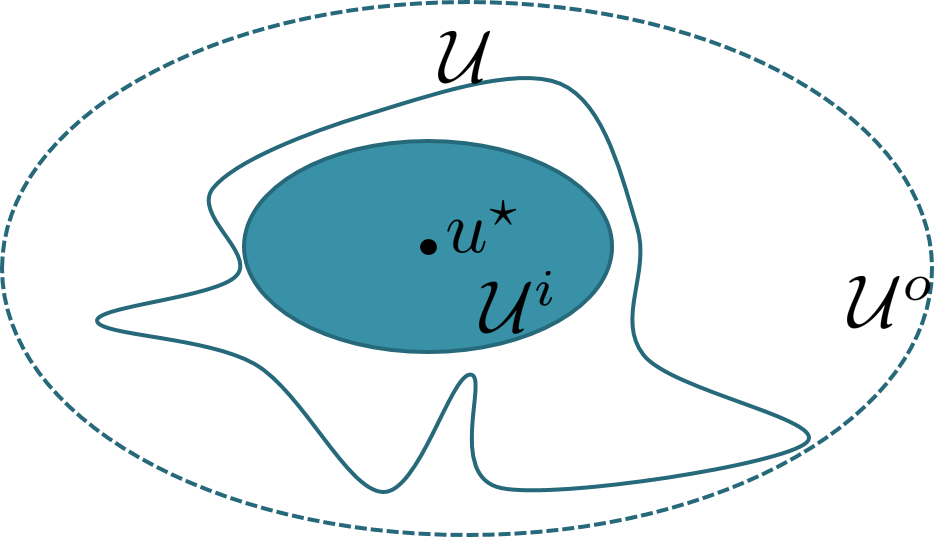}
\caption{Inner and outer approximations of $\Ucal$}\label{fig:InnerOuter}
\end{figure}

\subsection{Relationship to previous results on AC power flow}
In this section, we consider the special case of the AC power flow equations \eqref{eq:PF} and compare our result to the best-known previous result on inner approximations \cite{wang2016explicit}. We start with equation (4) from \cite{wang2016explicit}: $V=w+Z\inv{\mathrm{diag}\br{\herm{V}}}\herm{s}$ (this is simply a restatement of the power flow equations \eqref{eq:PF} introduced earlier with $Z=Y^{-1},w=-Y_{-1}Y_0$). We rewrite this system as a quadratic system in
$\gamma=\frac{w}{V}$ (we assume implicitly that we are not interested in solutions with $0$ voltage since this is not practical in a real power system):
\begin{align}
f\br{\gamma,s}=\diag{\gamma}\inv{\diag{w}} Z\inv{\herm{\diag{w}}}\diag{\herm{s}}\herm{\gamma}+\gamma-\mathbf{1}=0\label{eq:EPFLQuad}
\end{align}
where $\mathbf{1}$ is a vector with all entries equal to $1$ and $\diag{w}$ is a diagonal matrix with entries $w$. This is in the form studied in the complex extension of theorem \ref{thm:MainCert} (theorem \ref{thm:MainCertComplex} in appendix section \ref{sec:AppComplex}). Since the focus of this paper is creating a general framework to analyze quadratic systems and not on specific application to power systems, we do not provide details here. We simply state our final result on the solvability region computed by applying theorem \ref{thm:MainCertComplex} to \eqref{eq:EPFLQuad} with the nominal solution $\gamma^\star=\mathbf{1},s^\star=0$. Define $\zeta\br{s}=\inv{\diag{w}} Z\inv{\herm{\diag{w}}}\diag{\herm{s}}$. Then, \eqref{eq:EPFLQuad} has a solution if
\begin{align}
2\norm{\zeta\br{s}\mathbf{1}}_\infty+2\sqrt{\norm{\zeta\br{s}\mathbf{1}}_\infty\norm{\zeta\br{s}}_\infty} \leq 1 \label{eq:EPFLCertFinal}
\end{align}
If we use the bound $\norm{\zeta\br{s}\mathbf{1}}_\infty\leq \norm{\zeta\br{s}}_\infty$, this condition is implied by the condition $\norm{\zeta\br{s}}_\infty\leq \frac{1}{4}$ which is the region defined by corollary 1 in \cite{wang2016explicit}. Thus our analysis produces a stronger result than that from \cite{wang2016explicit}, at least for the particular case when $s^\star=0,\gamma^\star=\mathbf{1}$. Another advantage of our framework in this case is that we can obtain a region around any nominal solution $\br{\gamma^\star,s^\star}$ with nonsingular Jacobian, while \cite{wang2016explicit} requires a stronger assumption on the nominal solution (which can be stated as $\norm{\zeta\br{s^\star}}_\infty \leq \frac{1}{\norm{\gamma^\star}_\infty^2}$).

\subsection{Computational issues}
Thus far, we have assumed that the quantities $e\br{u-u^\star},h\br{u},g\br{u-u^\star}$ can be computed - in this section, we discuss the associated computational issues. We focus on the norm $\norm{\cdot}_\infty$ although similar ideas apply to other norms. $e\br{u-u^\star}$ is easy to compute since it simply involves taking the norm of an affine function of $u$. For any fixed value of $u$, $\inv{J_{\star\star}}J_\star\br{u;y}-y$ is a linear function of $y$ and is of the form $M^L\br{u}y$ where $M^L\br{u}$ is a matrix-valued function of $u$. Thus, $g\br{u-u^\star}=\norm{M^L\br{u}}_\infty$ where the norm stands for the induced matrix norm which is easy to compute. $h\br{u}$ involves maximizing the norm of a nonlinear and possibly nonconvex function of $x$, which is NP-hard in general. However, $h\br{u}$ can be written as $\max_{\norm{x}\leq 1}\norm{M^Q\br{u}\mathrm{vec}\br{xx^T}}_\infty$ where $\mathrm{vec}\br{xx^T}$ denotes all quadratic terms in $x$ ($x_1^2,x_1x_2,\ldots,x_1x_n,x_2^2,\ldots$) and $M^Q\br{u}$ is an affine matrix-valued function of $u$. Its easy to see that $\norm{\mathrm{vec}\br{xx^T}}_\infty\leq 1$ if $\norm{x}\leq 1$, so that $h\br{u}\leq \norm{M^Q\br{u}}_\infty$ showing that a bound on $h\br{u}$ can be computed easily. A tighter bound on $h\br{u}$ can be computed via convex relaxation - replacing the quadratic terms $xx^T$ with a symmetric matrix $X$, we can bound $h\br{u}$ by $\displaystyle\max_{X\succeq 0,\mathrm{diag}\br{X}\leq 1}\norm{M^Q\br{u}\mathrm{vec}\br{X}}_\infty$ which equals
$\displaystyle\max_{1\leq i \leq n}\max_{\sigma \in \{+1,-1\}} \displaystyle\max_{X\succeq 0,\mathrm{diag}\br{X}\leq 1}[ M^Q\br{u}\mathrm{vec}\br{X}]_i\sigma$
where the inner maximization can be computed by solving a semidefinite program (SDP) and the outer maximization requires enumerating $2n$ possibilities.

\section{Numerical results}\label{sec:NumRes}

In this section, we perform numerical studies applying the techniques from this paper to the AC power flow equations. 
We use the 18 bus distribution network taken from \cite{grady1992application} included with the {\rm MATPOWER} software \cite{zimmerman2011matpower}. 
We generate a random direction of perturbation $\hat{s}$ in the injection space and compute how far we can move along this direction before we violate \eqref{eq:EPFLCertFinal} - this is given by the largest scaling factor $t$ for which $t\hat{s}$ satisfies $\kappa\br{t\hat{s}}\leq 1$. We denote this value of $t$ by  $t_{\hat{s}}=\frac{1}{\kappa\br{\hat{s}}}$. We compare our results with the condition from \cite{wang2016explicit} (which in turn is known to generalize the results from \cite{yu2015simple} and \cite{bolognani2016existence}). When $\gamma^\star=\mathbf{1},s^\star=0$, the result corollary 1 in \cite{wang2016explicit} reduces to $\kappa^\prime(s)=4\norm{\zeta\br{s}}_\infty\leq 1$ so the farthest one can go in the direction $\hat{s}$ before hitting the boundary of this region is $t^\prime_{\hat{s}}=\frac{1}{\kappa^\prime\br{\hat{s}}}$. We use the ratio $\frac{t^\prime_{\hat{s}}}{t_{\hat{s}}}$
to compare our result \eqref{eq:EPFLCertFinal} and \cite{wang2016explicit}. If this number is larger than $1$, our certificate defines a larger region than \cite{wang2016explicit} along direction $\hat{s}$. 

We generate $10000$ random directions $\hat{s}$ and compute the ratio of $t_{\hat{s}}$ to $t^\prime_{\hat{s}}$. We plot a histogram of the ratio in figure \ref{fig:EPFLCompare} - the results show that our condition achieves an improvement of a factor of $2$ (our region is twice as large at least along the direction $\hat{s}$) or higher in most of the directions. We do a similar analysis comparing our certificate $t_{\hat{s}}$ to an upper bound on the largest scaling $t$ such that a solution to the equations \eqref{eq:EPFLQuad} does not exist. In order to obtain the upper bound, we consider a convex relaxation of \eqref{eq:EPFLQuad}:
\[t\br{\zeta\br{ \hat{s}}\odot \Gamma}\mathbf{1}+\gamma=\mathbf{1},\Gamma\succeq \herm{\gamma}\gamma^T\]
where $\odot$ denotes element-wise multiplication of matrices and $\succeq$ denotes inequality in the semidefinite sense. We define $t^\star_{\hat{s}}$ to be the smallest value of $t$ for which the relaxation has no solution (is infeasible). This number can be calculated by increasing the value of $t$ until the problem becomes infeasible. In figure \ref{fig:RelaxCompare}, we plot a histogram of $\frac{t_{\hat{s}}}{t^\star_{\hat{s}}}$ over $10000$ random directions $\hat{s}$. When the ratio is close to $1$, it indicates that the boundary of the region defined by \eqref{eq:EPFLCertFinal} is close to the true boundary (beyond which no solution exists to the PF equations). The results show that at least along some directions $\hat{s}$, the boundary of our certificate is close to the true boundary (the ratio is larger than $.8$) although in other directions it can be significantly away from the boundary. 
\begin{figure}[htb]
    \begin{subfigure}[b]{0.45\columnwidth}
\centering \includegraphics[width=.99\columnwidth]{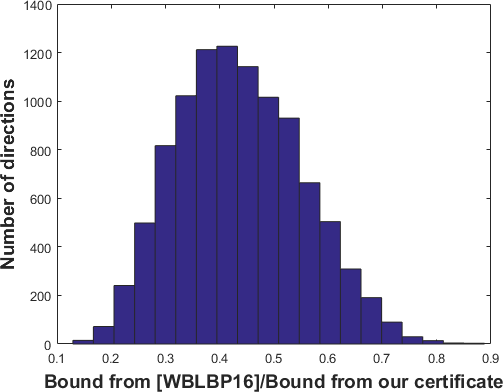}
\caption{Comparison of certificate \eqref{eq:EPFLCertFinal} and \cite{wang2016explicit} (histogram of $t^\prime_{\hat{s}}/t_{\hat{s}}$)}\label{fig:EPFLCompare}
    \end{subfigure}
        \begin{subfigure}[b]{0.45\columnwidth}
\centering \includegraphics[width=.99\columnwidth]{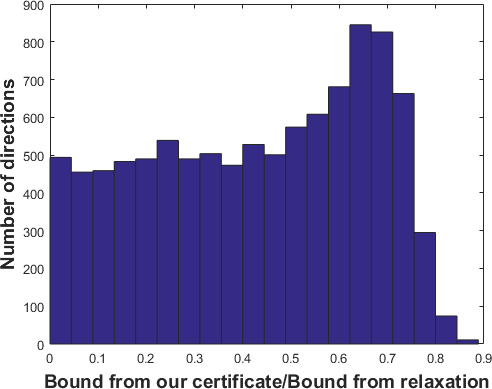}
\caption{Comparison of certificate \eqref{eq:EPFLCertFinal} and relaxation (histogram of $t_{\hat{s}}/t^\star_{\hat{s}}$)}\label{fig:RelaxCompare}
\end{subfigure}
\end{figure}

Finally, we plot a 2-dimensional project of the regions defined by $\kappa\br{s}\leq 1,\kappa^\prime\br{s}\leq 1$ and the relaxation. \begin{figure}
\centering
\includegraphics[width=.5\columnwidth]{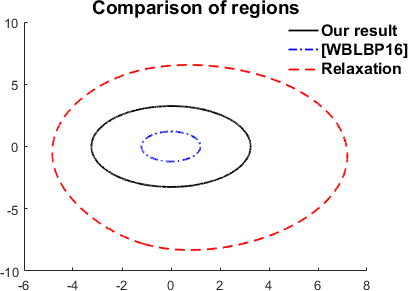}
\caption{2-dimensional projection of solvability regions}\label{fig:EPFLCompareRegions}
\vspace{-.7cm}
\end{figure}
We do this by picking a random direction $\hat{s}$, and rotating this direction by multiplying by a unit complex number: $\hat{s}\br{\theta}=\hat{s}\exp\br{\sqrt{-1}\theta}$ (with $\theta$ going from $0$ to $2\pi$).  For each value of $\theta$ we compute the maximum scalings $t_{\hat{s}\br{\theta}},t^\prime_{\hat{s}\br{\theta}},t^\star_{t_{\hat{s}\br{\theta}}}$ to determine the boundary of the region (see figure \ref{fig:EPFLCompareRegions}). The plot shows that our region is significantly larger than that from \cite{wang2016explicit} and close to the outer bound defined by the relaxation at least in some directions.


\section{Conclusions}\label{sec:Conc}
We have developed a general framework for computing inner approximations of the solvability regions of affinely parameterized quadratic systems of equations. The regions we construct are unions of convex sets in the space of parameters.  For the specific case of the AC power flow equations, we show that our work improves upon the best know bound from previous work \cite{wang2016explicit} both theoretically and numerically.

We hope that these results will serve as a basis for further investigations: We left the choice of norm open in this paper - preliminary results indicate that choosing the right norm is critical for constructing large inner approximations of the solvability region.  We will also explore applications of the regions to checking robust feasibility (with respect to uncertain parameters) and robust optimization (over controllable parameters) particularly in the context of infrastructure networks like power, water and gas grids. Generalizing the results of this paper to affinely parameterized nonlinear equations (beyond quadratic) and more general constraint sets on $x$ (beyond norm balls) is also an important and interesting direction for further work.

\bibliographystyle{alpha}
\bibliography{Ref}

\begin{thebibliography}{WBLBP16}

\bibitem[ACM84]{arnon1984cylindrical}
Dennis~S Arnon, George~E Collins, and Scott McCallum.
\newblock Cylindrical algebraic decomposition i: The basic algorithm.
\newblock {\em SIAM Journal on Computing}, 13(4):865--877, 1984.

\bibitem[BFLW09]{beelitz2009framework}
Thomas Beelitz, Andreas Frommer, Bruno Lang, and Paul Willems.
\newblock A framework for existence tests based on the topological degree and
  homotopy.
\newblock {\em Numerische Mathematik}, 111(4):493--507, 2009.

\bibitem[BZ16]{bolognani2016existence}
Saverio Bolognani and Sandro Zampieri.
\newblock On the existence and linear approximation of the power flow solution
  in power distribution networks.
\newblock {\em IEEE Transactions on Power Systems}, 31(1):163--172, 2016.

\bibitem[DT15]{dvijotham2015construction}
Krishnamurthy Dvijotham and Konstantin Turitsyn.
\newblock Construction of power flow feasibility sets.
\newblock {\em arXiv preprint arXiv:1506.07191}, 2015.

\bibitem[DVMC15]{DjGas}
Krishnamurthy Dvijotham, Marc Vuffray, Sidhant Misra, and Michael Chertkov.
\newblock Natural gas flow solutions with guarantees: {A} monotone operator
  theory approach.
\newblock {\em CoRR}, abs/1506.06075, 2015.

\bibitem[FHL07]{frommer2007proving}
Andreas Frommer, Fatmir Hoxha, and Bruno Lang.
\newblock Proving the existence of zeros using the topological degree and
  interval arithmetic.
\newblock {\em Journal of Computational and Applied Mathematics},
  199(2):397--402, 2007.

\bibitem[FK15]{franek2015robust}
Peter Franek and Marek Krcal.
\newblock Robust satisfiability of systems of equations.
\newblock {\em Journal of the ACM (JACM)}, 62(4):26, 2015.

\bibitem[FKW16]{franek2016robustness}
Peter Franek, Marek Krcal, and Hubert Wagner.
\newblock Robustness of zero sets: Implementation.
\newblock {\em Preprint http://www. cs. cas. cz/\~{}
  franek/rob-sat/experimental. pdf}, 2016.

\bibitem[GJ06]{goldsztejn2006inner}
Alexandre Goldsztejn and Luc Jaulin.
\newblock Inner and outer approximations of existentially quantified equality
  constraints.
\newblock In {\em International Conference on Principles and Practice of
  Constraint Programming}, pages 198--212. Springer, 2006.

\bibitem[GSN92]{grady1992application}
WM~Grady, MJ~Samotyj, and AH~Noyola.
\newblock The application of network objective functions for actively
  minimizing the impact of voltage harmonics in power systems.
\newblock {\em IEEE Transactions on Power Delivery}, 7(3):1379--1386, 1992.

\bibitem[Las09]{lasserre2009moments}
Jean~Bernard Lasserre.
\newblock {\em Moments, positive polynomials and their applications}, volume~1.
\newblock World Scientific, 2009.

\bibitem[Las15]{lasserre2015tractable}
Jean~B Lasserre.
\newblock Tractable approximations of sets defined with quantifiers.
\newblock {\em Mathematical Programming}, 151(2):507--527, 2015.

\bibitem[LB15]{louca2015acyclic}
Raphael Louca and Eilyan Bitar.
\newblock Acyclic semidefinite approximations of quadratically constrained
  quadratic programs.
\newblock In {\em American Control Conference (ACC), 2015}, pages 5925--5930.
  IEEE, 2015.

\bibitem[LB16]{louca2016stochastic}
Raphael Louca and Eilyan Bitar.
\newblock Stochastic ac optimal power flow with affine recourse.
\newblock In {\em Decision and Control (CDC), 2016 IEEE 55th Conference on},
  pages 2431--2436. IEEE, 2016.

\bibitem[WBLBP16]{wang2016explicit}
Cong Wang, Andrey Bernstein, Jean-Yves Le~Boudec, and Mario Paolone.
\newblock Explicit conditions on existence and uniqueness of load-flow
  solutions in distribution networks.
\newblock {\em IEEE Transactions on Smart Grid}, 2016.

\bibitem[YNT15]{yu2015simple}
Suhyoun Yu, Hung~D Nguyen, and Konstantin~S Turitsyn.
\newblock Simple certificate of solvability of power flow equations for
  distribution systems.
\newblock In {\em Power \& Energy Society General Meeting, 2015 IEEE}, pages
  1--5. IEEE, 2015.

\bibitem[ZMST11]{zimmerman2011matpower}
Ray~Daniel Zimmerman, Carlos~Edmundo Murillo-S{\'a}nchez, and Robert~John
  Thomas.
\newblock Matpower: Steady-state operations, planning, and analysis tools for
  power systems research and education.
\newblock {\em IEEE Transactions on power systems}, 26(1):12--19, 2011.

\end{thebibliography}
\section{Appendix}
\subsection{Proof of lemma \ref{lem:MainCert} and theorem \ref{thm:MainCert}}\label{sec:ProofMainCert}
\begin{proof}
Let $\alpha=f\br{x^\star,u}$. We are interested in finding conditions on $u$ that guarantee the existence of solution to $Q\qud{x}{x}{u}+L\br{u}x+K\br{u}=0$. Using $\alpha=f\br{x^\star,u}$, we obtain
$Q\qud{x}{x}{u}-Q\qud{x^\star}{x^\star}{u}+L\br{u,x-x^\star}+\alpha=0$. Let $dx=x-x^\star$. Then, we obtain
\begin{align}
& 2Q\qud{x^\star}{dx}{u}+Q\qud{dx}{dx}{u}+L\br{u,dx}+\alpha =0 \nonumber \\
&=Q\qud{dx}{dx}{u}+J\br{dx,u}+\alpha\nonumber \\
&=Q\qud{dx}{dx}{u}+J\br{dx,u}-J\br{dx,u^\star}+\alpha+J_{\star\star} dx\nonumber
\end{align}
Rescaling the above equation by $\inv{J_{\star\star}}$, we define $G\br{dx}$ to equal:
$-\inv{J_{\star\star}}\br{J\br{dx,u}-J\br{dx,u^\star}+Q\qud{dx}{dx}{u}+\alpha}$
so that $f\br{x,u}=0$ is equivalent to $G\br{dx}=dx$. This equation is in fixed point form. We establish that $G$ maps the set $\{dx:\norm{dx}\leq r\delta\}$ onto itself and apply Brouwer's theorem (theorem \ref{thm:Brouwer}). $\norm{G\br{dx}}$ is bounded by
$\max_{\norm{dx}\leq r\delta}\norm{Q\qud{dx}{dx}{u}}+r\delta g\br{u-u^\star}+e\br{u-u^\star}$ which is smaller than $\delta^2 h\br{u}+\delta g\br{u-u^\star}+e\br{u-u^\star}$. We need $\norm{G\br{dx}}\leq r\delta$ for Brouwer's theorem to apply, or $(r\delta)^2 h\br{u}+r\delta g\br{u-u^\star}+e\br{u-u^\star}	\leq r\delta$. Dividing both sides by $\delta$, we obtain $r\delta h\br{u}+ g\br{u-u^\star}+\frac{1}{r\delta}e\br{u-u^\star}	\leq 1$.
This is precisely the condition that defines the set $\Ucal^i_r(\delta)$ in the theorem. Since we are required to find a solution with $\norm{dx}\leq r$ and the above condition guarantees existence within $\norm{dx}\leq r\delta$, for any value $\delta \in (0,1]$, every $u\in\Ucal^i_r(\delta)$ is guaranteed to have a solution with $\norm{dx}\leq r$.

Since $h,g,e$ are norms of affine functions of $u$ (or maximum over norms), each $u \in \Ucal^i_r(\delta)$ is convex. Further, choosing $\delta=\delta(u-u^\star)=\min(\frac{2e\br{u-u^\star}}{r},1)$, we know that a $u\in\Ucal^i_r$ if $r\delta(u-u^\star) h(u)+g\br{u-u^\star}\leq \frac{1}{2},e\br{u-u^\star}\leq \frac{r}{2}$. At $u=u^\star$, the LHS of both inequalities evaluates to $0$ and hence $u^\star \in \Ucal^i_r$. Further, since the LHS is a continuous function of $u$, there is a ball $\{u:\norm{u-u^\star} \leq t\}$ that is still contained in $\Ucal^i_r$. Hence $u^\star \in \mathrm{Int}\br{\Ucal^i_r}$.
\end{proof}

\subsection{Proof of theorem \ref{thm:Converse}}\label{App:PfConverse}
\begin{proof}
We apply theorem \ref{thm:MainCert} to the special case \eqref{eq:QuadSimple} - in this case $h\br{u}=h^\star=\max_{\norm{x}\leq 1}\norm{Q\br{x,x}}$ and $J$ does not depend on $u$, hence $g\br{u-u^\star}=0$. Thus, the condition for existence of a solution within the set $\norm{x-x^\star}\leq \delta$ becomes $\delta h^\star + \frac{e\br{u-u^\star}}{\delta}\leq 1$. We can choose $\delta=\frac{\kappa}{2h^\star}$ to obtain $
e\br{u-u^\star}\leq  \frac{2\kappa-\kappa^2}{4h^\star}$
to guarantee existence of a solution within $\mathcal{B}_\kappa$.
Conversely, suppose that $e\br{u-u^\star}> \frac{2\kappa+\kappa^2}{4h^\star}$ and that \eqref{eq:QuadSimple} has a solution in $\mathcal{B}_\kappa$. Using the derivation of the previous theorem, we arrive at $-\inv{J_{\star\star}}Q\br{dx,dx}-\inv{J_{\star\star}}\alpha=dx$.
Taking norms and applying the triangle equality, we get $
\norm{dx}+\norm{\inv{J_{\star\star}}Q\br{dx,dx}}\geq e\br{u-u^\star}$.
If $\norm{dx}\leq \frac{\kappa}{2h^\star}$, then we know that
$\norm{\inv{J_{\star\star}}Q\br{dx,dx}}\leq \frac{\kappa^2}{4h^\star}$.
Thus, we get $\frac{2\kappa+\kappa^2}{4h^\star}\geq e\br{u-u^\star}$
which is a contradiction to our assumption. Thus, if $e\br{u-u^\star}>\frac{\kappa+\kappa^2}{h^\star}$, there is no solution to \eqref{eq:QuadSimple} in $\mathcal{B}_\kappa$.
\end{proof}
\subsection{Complex equations}\label{sec:AppComplex}
We extend theorem \ref{thm:MainCert} to the complex case:
\begin{align}
f\br{x,u}=Q\qud{x}{\herm{x}}{u}+L\br{x,u}+K\br{u}=0\label{eq:QuadComplex}
\end{align}
where $Q,L,K$ are complex-valued trilinear/bilinear/linear forms in the complex variables $x \in \Com^n,u\in \Com^k$. Define
\begin{align*}
& J\br{x,y,u}=\left.\frac{\partial f}{\partial x}\right|_{x,u}y+\left.\frac{\partial f}{\partial \herm{x}}\right|_{x,u}\herm{y}, J_\star\br{y,u}=J\br{x^\star,y,u} \\
& J_{\star\star}=\begin{pmatrix}
                                                     \left.\frac{\partial f}{\partial x}\right|_{x^\star,u^\star} & \left.\frac{\partial f}{\partial \herm{x}}\right|_{x^\star,u^\star} \\
                                                     \vspace{-.3em} & \vspace{-.3em} \\
                                                     \herm{\left.\frac{\partial f}{\partial \herm{x}}\right|_{x^\star,u^\star}} & \herm{\left.\frac{\partial f}{\partial x}\right|_{x^\star,u^\star}}
                                                   \end{pmatrix},
 \inv{J_{\star\star}}=\begin{pmatrix}
                         M_\star & N_\star \\
                         \herm{N_\star} & \herm{M_\star}
                       \end{pmatrix}
                       \end{align*}
\begin{theorem}\label{thm:MainCertComplex}
Let $\br{x^\star,u^\star}$ satisfy \eqref{eq:QuadComplex} and suppose that $J_{\star\star}$ be invertible. Define the following quantities:
\begin{align*}
& e\br{u-u^\star}  =\norm{M_\star f\br{x^\star,u}+N_\star \herm{f\br{x^\star,u}}} \\
& h\br{u} =\max_{\norm{y}\leq 1} \norm{M_\star Q\qud{y}{y}{u}+N_\star\herm{Q\qud{y}{y}{u}}} \\
& g\br{u-u^\star}=\max_{\norm{y}\leq 1} \norm{M^\star J\br{y,u}+N^\star\herm{J\br{y,u}} -y}
\end{align*}
Define the sets $\Ucal^i_r\br{\delta},\Ucal^i_r,\Ucal^i_\infty$ as in theorem \ref{thm:MainCert}. Then for each $u \in \Ucal^i_r$, there is a solution to \eqref{eq:QuadComplex} with $\norm{x-x^\star}\leq r$. Further, $\Ucal^i_r(\delta)$ is a convex set and hence $\Ucal^i_r$ is a union of convex sets with $u^\star \in \mathrm{Int}\br{\Ucal^i_r}$. For each $u\in\Ucal^i_\infty$, there exists a solution $x$ to $f\br{x,u}=0$.
\end{theorem}
\begin{proof}
Similar to theorem \ref{thm:MainCert}.
\end{proof}

\end{document}